\documentclass[12pt]{amsart}

\newtheorem{theorem}{Theorem}[section]

\newtheorem{corollary}[theorem]{Corollary}

\newtheorem{lemma}[theorem]{Lemma}

\theoremstyle{definition}

\newtheorem{definition}[theorem]{Definition}

\newtheorem{example}[theorem]{Example}

\theoremstyle{parrafo}

\begin{document}

\title[]{On the pointwise domination of a function by its maximal function}

\author{J. M. Aldaz}
\address{Instituto de Ciencias Matem\'aticas (CSIC-UAM-UC3M-UCM) and Departamento de 
Matem\'a\-ticas,
Universidad  Aut\'onoma de Madrid, Cantoblanco 28049, Madrid, Spain.}
\email{jesus.munarriz@uam.es}
\email{jesus.munarriz@icmat.es}

\thanks{2000 {\em Mathematical Subject Classification.} 42B25}

\thanks{The author was partially supported by Grant MTM2015-65792-P of the
MINECO of Spain, and also by by ICMAT Severo Ochoa project SEV-2015-0554 (MINECO)}







\begin{abstract} We show that under rather general circumstances, 
the almost everywhere pointwise inequality $|f|(x)  \le Mf (x)$ is equivalent to a 
weak form of the conclusion of the Lebesgue density theorem, for
totally bounded closed sets.  We derive both positive and negative results from this characterization.
\end{abstract}


\maketitle


\markboth{J. M. Aldaz}{Pointwise domination by the maximal function}

\section {Introduction} Let   $f$ be a locally integrable function, and let  $Mf$ denote the centered
Hardy-Littlewood maximal function of $f$.
A standard proof of the almost everywhere pointwise inequality $|f|(x)  \le Mf (x)$  runs as follows: Prove a weak type
(1,1) inequality for $Mf$, deduce the Lebesgue differentiation theorem, and use the fact that the supremum over all balls 
 is at least  as large as the limit when $r\downarrow 0$. 
 
 To a considerable extent,  the usefulness of $Mf$ in the presence of
   weak type (1,1) bounds, comes from the
 fact that it is a.e. larger than $|f|$ but not much larger, since it is still comparable 
 to $f$ in the $L^p$ sense, for $p > 1$. One might
 suspect than when the weak type (1,1) bounds fail, $Mf$ should be exceedingly large, so
 the a.e. pointwise inequality $|f|(x)  \le Mf (x)$ ought to hold. But as we shall see, this
 need not be the case.
 We show below that the inequality 
$|f|(x)  \le Mf (x)$ a.e.
 is equivalent, under 
very general assumptions, to the following  weak form of the conclusion of  the Lebesgue density theorem:
 For 
all closed and totally bounded sets $F\subset \operatorname{supp}\mu$ and a.e. $x$,
$$\limsup_{r\downarrow 0} \frac{1}{\mu
( B(x, r))} \int _{ B(x, r )}  \mathbf{1}_F\ d\mu = \mathbf{1}_F (x).
$$
It then follows from  a result of D. Preiss 
 that  there exists a Gaussian measure on an infinite dimensional Hilbert space $H$ for which 
$|f|(x)  \le Mf (x)$ fails on a set of positive measure, while by results of D. Preiss  and J. Ti\u{s}er
the said inequality holds a.e. for some classes of Gaussian measures on $H$  (cf.  
\cite{Pr}, \cite{PrTi}, \cite{Ti}).

 We also prove some positive results in the setting of metric measure spaces:
If the space is geometrically doubling, or (without restrictions on the metric space) 
if the size of balls is not ``too small" as the radii tend to zero,
then  $|f|(x)  \le Mf (x)$ holds  a. e., for all reasonable (meaning $\tau$-additive and finite on bounded sets) 
 Borel measures. The last section deals with $\sigma$-algebras larger 
 than the Borel sets, along the lines of  \cite{AlMi} on 
  extensions of the Lebesgue
 differentiation theorem.
 
 I am indebted to an anonymous referee for his or her very careful reading of this paper,
 as well as several useful suggestions.

\section {Some definitions and  background}

First  we present the
basic definitions and background material, spending some time on why the definition of metric measure spaces given
below, involving  $\tau$-additivity,  is the right one in this context.
 The standard axioms for set theory, Zermelo-Fr\"ankel with Choice, are abbreviated
by ZFC. 
We  use $B(x,r) := \{y\in X: d(x,y) < r\}$ to denote open balls, 
 and 
$B^{cl}(x,r) := \{y\in X: d(x,y) \le r\}$ to refer to metrically closed balls.  It is always assumed that measures are not
identically 0.

\begin{definition} A measure on a topological space is {\em Borel} if it is
	defined on the sigma algebra generated by the open sets. 
	A Borel measure is  {\em $\tau$-smooth} 
or {\em $\tau$-additive}, if for every
collection  $\{U_\alpha : \alpha \in \Lambda\}$
 of  open sets, 
$$
\mu (\cup_\alpha U_\alpha) = \sup_{\mathcal{F}} \mu(\cup_{i=1}^nU_{\alpha_i}),
$$
 where the supremum is taken over all finite subcollections $\mathcal{F} = \{U_{\alpha_1}, \dots, U_{\alpha_n} \}$
of  $\{U_\alpha : \alpha \in \Lambda\}$.
 We say that $(X, d, \mu)$ is a {\em metric measure space} if
$\mu$ is a  $\tau$-additive  Borel measure on the metric space $(X, d)$, such that $\mu$ assigns finite measure
to bounded Borel sets. 
\end{definition}

Recall that via the Caratheodory's outer measure construction, any measure $\mu$ on
a $\sigma$-algebra can be extended in a unique way to its $completion$, defined on
the $\sigma$-algebra of all $\mu^*$-measurable subsets. By a slight abuse of
notation, we also call $(X, d, \mu)$ a metric measure space
when the Borel measure $\mu$ is replaced by its completion, denoted
again  by $\mu$
(this being motivated by the fact that the equivalence classes of measurable
functions in the larger $\sigma$-algebra, still have Borel representatives). 

The following lemma is well known, see for instance
\cite[Theorem 7.1.7]{Bo}. It is placed here as a reminder.
Recall that a set $T$ is {\em totally bounded} if given any $r > 0$, $T$ can be covered
with finitely many balls of radius $r$.

\begin{lemma} \label{inner}  Let $(X, d)$ be a metric  space endowed with a finite Borel measure
$\mu$. Then every Borel set can be approximated in measure from within
by closed sets.
Furthermore, if $\mu$ is
$\tau$-additive, then the approximating closed sets can be chosen to be totally bounded. 
\end{lemma}

\begin{proof}
 Call a set $A$ {\em approximable} if for every $\varepsilon > 0$
there exist an open set $O$ and a closed set $C$ such that $C \subset A \subset O$ and $\mu (O \setminus C) < \varepsilon$;
clearly,  the closed sets $C$  are approximable (since $C = \cap_n \{x\in X: d(x, C) < 1/n\}$), the complement of an approximable set is
approximable, and countable unions of approximable sets are approximable. 
Thus,  all Borel sets are approximable.
We give more detail on why countable unions of approximable sets are approximable from within. Given $\varepsilon > 0$,
and a sequence $\{A_n\}_{1 \le n < \infty}$ of approximable sets, first reduce it,
using countable additivity, 
to a finite sequence $\{A_n\}_{1 \le n \le N}$ satisfying 
$\mu \cup_{1 \le n < \infty}A_n - \mu \cup_{1 \le n \le N} A_n < \varepsilon/4$. 
Then select, for each $n$, a closed set $C_n \subset A_n$ such that
$\mu (A_n \setminus C_n ) < \varepsilon/2^{n + 2}$, and note that
$\cup_{1 \le n \le N} C_n$ is a closed set with 
$\mu \cup_{1 \le n < \infty} A_n - \mu \cup_{1 \le n \le N} C_n < \varepsilon/2$.

Assume next that $\mu$ is $\tau$-additive. Let $C$ be closed, and let $\varepsilon > 0$. Select a finite subcollection 
$B_1^1, \dots, B_{n_{1}}^1$ from the cover of $C$ given by $\{B(x,1): x \in C\}$
(by $\tau$-additivity, since this cover may be uncountable) such that
$\mu (C\setminus \cup_1^{N_1}  B_i^1 ) < \varepsilon/2$, and set $O_1 =  \cup_1^{N_1}  B_i^1$.
 Then repeat, using succesively balls of radii $2^{-1}, 2^{-2}, \dots$,
so that at stage $k$, $O_k$ is a finite union of balls of radius $2^{-k}$, centered at points of $C$, and satisfying
$\mu (C\setminus O_k ) < \varepsilon/2^{k}$. Then $K := \cap_n \overline{O_n} \subset C$ is closed,  totally bounded, and satisfies
$\mu (C\setminus K ) < \varepsilon$. 
\end{proof}

From now on, we assume that all Borel measures under consideration are finite on bounded sets. The role of $\tau$-additivity can be explained as follows: In general, uncountable unions of 
measurable sets need not be measurable, but when the measurable sets are open, then
the uncountable unions are open, and hence Borel. Thus, $\tau$-additivity is a strengthening
of countable additivity for open sets, which allows us to reduce arbitrary unions to finite
unions with  arbitrarily small errors.

 If the metric space is separable, then uncountable unions
of open sets can be reduced to countable unions, and hence $\tau$-additivity holds for
all Borel measures $\mu$. 
Separability is sometimes assumed when defining metric measure spaces (cf. for
instance \cite[p.  62]{HKST}) but this has the unfortunate consequence of
excluding from the definition spaces that appear naturally in analysis, such as $L^\infty ([0,1])$.

Also, if $X$ is not separable
but $\mu$ is Radon (inner regular with respect to the compact sets) then $\mu$ is $\tau$-additive. Of course, it is natural to ask whether there are Borel measures on metric spaces
that fail to be $\tau$-additive. More on this to follow soon.

\begin{definition}\label{maxfun} Let $(X, d, \mu)$ be a metric measure space and let $g$ be  a locally integrable function 
on $X$. For each  $x\in X$, {\em the centered Hardy-Littlewood maximal operator} $M_{\mu}$ 
 is given by
\begin{equation}\label{HLMFc}
M_{\mu} g(x) := \sup _{\{r :  \ 0 < \mu (B^{cl}(x, r))\}}  \frac{1}{\mu
(B^{cl}(x, r))} \int _{B^{cl}(x, r)}  |g| d\mu.
\end{equation}
When the supremum is taken over radii $r\le R$, for some fixed $R > 0$, we obtain the {\em localized maximal operator  }
\begin{equation}\label{HLMFcR}
M_{\mu, R} \  g(x) := \sup _{\{r \le R : \ 0 < \mu (B^{cl}(x, r))\}}  \frac{1}{\mu
(B^{cl}(x, r))} \int _{B^{cl}(x, r)}  |g| d\mu.
\end{equation}
\end{definition}

Often we simplify notation by writing $M_{R}$ and $M$  instead of $M_{\mu, R}$ and $M_{\mu}$, when no confusion is likely
to arise.

Maximal operators can be defined using open balls instead of closed balls,
and this does not change their values, for we are taking suprema.
The choice made here is merely  for convenience, since from the viewpoint of the covering arguments
given below, it is better to have closed balls.
In fact, we shall utilize the definition with open balls whenever it suits us. 

Recall that 
the complement of
the support $(\operatorname{supp}\mu)^c := \cup \{ B(x, r): x \in X, \mu B(x,r) = 0\}$
of a Borel  measure
 $\mu$,  is an open set, and hence measurable. 
By the convention used in (\ref{HLMFc}), the  Hardy-Littlewood maximal function $M_{\mu} g$ is well defined
everywhere on $X$, even on $(\operatorname{supp}\mu)^c$.  However,  if there exists a ball $B(x, S) \subset (\operatorname{supp}\mu)^c$ 
with $R < S$, then $M_{\mu, R} g$ is not defined at $x$.
Another possible convention is to leave  $M_{\mu, R}$ and $M_{\mu}$  undefined off 
$\operatorname{supp}\mu$. 
Of course, which convention one uses matters little if $(\operatorname{supp}\mu)^c$ has measure zero. When $\mu (X \setminus \operatorname{supp}\mu) = 0$ 
we say $\mu$ has {\em full support}. 

Now $\tau$-additivity implies that 
$\mu$ has full support,
since $X \setminus \operatorname{supp}\mu $ is a union of open balls of measure zero.
Actually, the other implication also holds, for the support is always separable; otherwise,
we would be able to find an uncountable collection of disjoint open balls of positive measure,
contained in some larger ball (of finite measure).
Thus,  having full support is equivalent to
$\tau$-additivity. 

From the viewpoint of defining averaging (for a fixed radius)
 and maximal averaging operators,
it is important to have both full support and balls with finite measure. Thus, 
 metric measure spaces, as defined above, constitute a natural and sufficiently  general class 
 to deal with the type of issues addressed here.
 
Returning to the question whether there are Borel measures on metric spaces
that fail to be $\tau$-additive, a positive answer turns out to be equivalent to the assumption of
the existence of certain very large cardinals.

\begin{definition} Let $(X,d)$ be a metric space with $d(x,y) = 1$ if $x\ne y$.  Suppose there exists a $0-1$ valued Borel measure 
on $X$, such that for all $x\in X$, $\mu\{x\} = 0$ (and $\mu X = 1$). Then we say that $X$ has {\em measurable cardinality}.
\end{definition} 

It is well known that the existence of  measurable cardinals entails that ZFC has models, and hence by G\"odel's second
incompleteness theorem, measurable cardinals cannot be proven to exist within ZFC (unless ZFC is inconsistent). 
Thus, in standard mathematical practice we will never meet a measurable cardinal
(the reader interested in obtaining additional information on these subjects, may want to consult 
\cite[Theorem  1.12.44]{Bo}, and for a more detailed treatment, \cite[Chapters 12 and 23]{JuWe}).

 If
ZFC is consistent, then we can add as a new axiom the non-existence of measurable cardinals and
obtain a new consistent theory, within which all Borel measures on a metric space are $\tau$-additive,
cf. \cite[Proposition 7.2.10]{Bo}. Thus, we can consistently assume that every  metric space with a Borel
measure (finite on bounded sets)  is actually a metric measure space as defined above.

Furthermore, if the existence of measurable cardinals is assumed, the connection between denstity results
and  pointwise domination by the maximal function breaks down, since it may be impossible to consider balls
of small radii.

\begin{example} \label{noncentered} Suppose $X$ is a measurable cardinal, and let $d$ be the discrete distance 
on $X$ ($d(x,y) = 1$  if $x\ne y$). Let $\mu$ be a $0-1$ Borel measure on $X$ that vanishes on singletons, and
let $y\in X$. Now consider $\mathbf{1}_{X\setminus \{y\}}$. Define $\nu := \mu + \delta_y$. 
With the convention from (\ref{HLMFc}), the only 
closed balls  centered at points of $X\setminus \{y\}$ we can consider 
have radius at least 1, and hence they contain the whole space. Thus, on $X\setminus \{y\}$ we have
 $M_\nu \mathbf{1}_{X\setminus \{y\}} = 1/2 < \mathbf{1}_{X\setminus \{y\}}$. 
\end{example} 

\section {Main result}

\begin{theorem}\label{equivalence} Let $(X, d, \mu)$ be a metric measure space. The
following are equivalent:

1)  For all $f\in L_{loc}^1(\mu)$,
$Mf (x) \ge |f|(x)$ almost everywhere.  

2) For 
all closed and totally bounded sets $F\subset \operatorname{supp}\mu$,  we have 
$$\limsup_{r\downarrow 0} \frac{1}{\mu
( B^{cl}(x, r))} \int _{ B^{cl}(x, r )}  \mathbf{1}_F\ d\mu = \mathbf{1}_F (x)
$$
almost everywhere.

3)  For all $f\in L_{loc}^1(\mu)$ and all $R >0$,
$M_R f (x) \ge |f|(x)$ almost everywhere.

\noindent Furthermore, the same equivalence holds if we use open instead of closed balls.
\end{theorem}

\begin{proof} While the present proof is written in terms of closed balls, it is easy to check that 
similar arguments  work for open balls. Alternatively, one may notice that neither the
value of the maximal function nor of the limsup are changed if we use open instead of 
closed balls.

 First, note that  3) implies 1) always. The hypothesis of 
 $\tau$-additivity is used as follows.
 By disregarding
a set of measure zero if needed,  we suppose that
$\operatorname{supp}\mu =X$. This entails that the definition of the maximal function involves balls of
all radii at every point, removing the obstacle from the preceding example.

Let us show that 1) implies 2). Suppose $F$ is closed and totally bounded, and assume 1). Then for almost every $x\in F$, we have $M \mathbf{1}_F (x)
 = \mathbf{1}_F (x)$. Fix one such $x$. If the supremum is attained at some fixed ball $ B^{cl}(x, R)$, so
 $$
\frac{1}{\mu ( B^{cl}(x, R))} \int _{ B^{cl}(x, R )}  \mathbf{1}_F\ d\mu = \frac{\mu (F \cap  B^{cl} (x, R))}{\mu ( B^{cl}(x, R))} 
 = 1,
 $$
 then $\mu ( B^{cl}(x, R) \setminus  F) = 0$, and hence for every $0 < r \le R$, 
$\mu ( B^{cl}(x, r) \setminus  F) = 0$. Therefore, 
$$
\lim_{r\downarrow 0} \frac{\mu (F \cap  B^{cl} (x, r))}{\mu
( B^{cl}(x, r))} = 1.
$$

Suppose next that the supremum is not attained. Select a sequence of radii 
$\{r_n\}_{n\ge 1}$ such that the averages approach 1 as $n \to \infty$.
By passing to a subsequence, we may assume that the radii either diverge
to infinity, or tend to some $0 < r_0 < \infty$, or go to zero. Next we prove
that the first two possibilities lead to contradictions.

Note first that if for some $0 < r_0 < \infty$ and some sequence $\{r_n\}_{n\ge 1}$
with $\lim_n r_n = r_0$ we have
$$
1 = M \mathbf{1}_F (x)
 = 
\lim_{n\to \infty} \frac{1}{\mu ( B^{cl}(x, r_n))} \int _{ B^{cl}(x, r_n )}  \mathbf{1}_F\ d\mu,
$$
then the supremum is attained at  $ B^{cl}(x, r_j)$ for some $j \ge 0$. 
To see why, observe that if $r_n \ge r_0$
for infinitely many $n$, by passing to a subsequence, we may assume that $r_n\downarrow r_0$,
and then
$$
\frac{1}{\mu ( B^{cl}(x, r_0))} \int _{ B^{cl}(x, r_0 )}  \mathbf{1}_F\ d\mu = 
\lim_{n\to \infty} \frac{1}{\mu ( B^{cl}(x, r_n))} \int _{ B^{cl}(x, r_n )}  \mathbf{1}_F\ d\mu = 1,
$$
so we can take $j = 0$.
On the other hand, if  $r_n \le r_0$ for all but (perhaps) finitely many $n$'s, we can
select a subsequence, also denoted by  $\{r_n\}_{n\ge 1}$, with 
 $r_n\uparrow r_0$. 
Then we can take $j = 1$, for if $\mu (B(x, r_1)  \setminus F) > 0$, the following
contradiction is obtained:
$$
1 = \lim_{n\to \infty} \frac{1}{\mu ( B^{cl}(x, r_n))} \int _{ B^{cl}(x, r_n )}  \mathbf{1}_F\ d\mu
\le
1 - \frac{\mu ( B^{cl}(x, r_1)\setminus F)}{\mu B(x, r_0 )} 
 < 1.
$$

Second, we show that the case $r_n \to \infty$ cannot happen either, so we must
have 
$$
1 = M \mathbf{1}_F (x) = \limsup_{r\downarrow 0} 
\frac{\mu (F \cap  B^{cl} (x, r))}{\mu
( B^{cl}(x, r))},
$$ 
 from which 2) follows, since $F$ is closed. To this end, we prove that 
$$
M \mathbf{1}_F (x) >  \limsup_{r\uparrow \infty} 
\frac{\mu (F \cap  B^{cl} (x, r))}{\mu
( B^{cl}(x, r))}.
$$ 
If $\mu X = \infty$, this follows from the fact that $\mu F < \infty$. And if 
$\mu X < \infty$, we select any $R > 0$ and note that by hypothesis,
$\mu ( B^{cl}(x, R) \setminus  F) > 0$. Now
$$
\limsup_{r\uparrow \infty} 
\frac{\mu (F \cap  B^{cl} (x, r))}{\mu
( B^{cl}(x, r))}
 \le 
\limsup_{r\uparrow \infty} 
\frac{\mu ( B^{cl} (x, r)) -  \mu ( B^{cl}(x, R) \setminus  F)}{\mu
( B^{cl}(x, r))} 
$$
$$
=  1 - \frac{\mu ( B^{cl}(x, R) \setminus  F)}{\mu
(X)} < 1.
$$

Next, assume 2). Let $f\ge 0$ be locally integrable. To obtain 3) it is enough to show that for every
$y\in X$ and every $R > 0$,  we have
$
M_R f (x) \ge  f(x)
$ almost everywhere on $B(y, S)$, where $S \ge R$.

 Choose $\varepsilon > 0$. Using Lemma \ref{inner}, we select a simple function $0 \le s \le  f \mathbf{1}_{B(y, 2 S)}$
of the form $s = \sum_1^n a_i \mathbf{1}_{F_i}$, where the sets $F_i$ are disjoint, closed and
totally bounded,  the coefficients $a_i$,
strictly positive, and furthermore, $ f \mathbf{1}_{B(y, 2 S)} < s + \varepsilon $, save perhaps on a subset of measure $< \varepsilon$.

We present more detail on why simple functions can be chosen in this way.
 First, by disjointification, a non-negative simple function $h$ can always  be expressed as 
 $h = \sum_1^n a_i \mathbf{1}_{D_i}$, where the sets $D_i$ are disjoint and the coefficients
 $a_i >0$. Next, given any
 $\delta > 0$, by Lemma \ref{inner}, for each $i\in {1,\dots, n}$ we can select $F_i \subset D_i$
 with $\mu (D_i \setminus F_i) < \delta/(n \max\{a_1, \dots, a_n\})$, so writing 
 $s = \sum_1^n a_i \mathbf{1}_{F_i}$, we have $\|h - s\|_1 < \delta$, from whence it follows
 that  this smaller
 class of simple functions is still dense in $L^1$. Finally, since $ f \mathbf{1}_{B(y, 2 S)} $
 is in $L^1$, we can find  a sufficiently large constant $T$ such that $f \mathbf{1}_{B(y, 2 S)} $
 and $g:= \min\{T,  f \mathbf{1}_{B(y, 2 S)}\}$ are equal save perhaps on a set of measure
 less than $\varepsilon$. Since the latter function is bounded, non-negative and in $L^1$,
 a simple function $s$ of the prescribed form can be found so that $s \le g < s + \varepsilon$
 everywhere.

Now if $s(x) > 0$, then there is a unique $i = i(x)$ such that $x \in F_i$. Since $\cup_{j\ne i} F_i$ is a finite union of closed sets,
it is closed, so the distance between $x$ and $\cup_{j\ne i} F_i$ is strictly positive. It follows that
$$
 \limsup_{r\downarrow 0} \frac{1}{\mu
( B^{cl}(x, r))} \int _{ B^{cl}(x, r )} s \ d\mu 
=
 a_i \limsup_{r\downarrow 0}  \frac{1}{\mu
( B^{cl}(x, r))} \int _{ B^{cl}(x, r )}  \mathbf{1}_{F_i}\ d\mu,
$$
so  for almost all  $x \in B(y, S)$ we have
$$
 M_R f (x) \ge  M_R s(x) \ge \limsup_{r\downarrow 0} \frac{1}{\mu
( B^{cl}(x, r))} \int _{ B^{cl}(x, r )} s \ d\mu $$
$$ =  \sum_1^n a_i \limsup_{r\downarrow 0}  \frac{1}{\mu
( B^{cl}(x, r))} \int _{ B^{cl}(x, r )}  \mathbf{1}_{F_i}\ d\mu
=  \sum_1^n a_i  \mathbf{1}_{F_i}(x)  = s(x).
$$
Thus, on $B(y, S)$ we have that $
 M_R f  > f - \varepsilon$, save perhaps on a set of measure $< \varepsilon$, and now 3) follows
by first  letting  $\varepsilon \downarrow 0$, and then $S\uparrow\infty$.
\end{proof}

\section{Corollaries}

We begin with a negative consequence of the characterization presented in Theorem \ref{equivalence}.

\begin{corollary}\label{preiss} There is a complete  metric measure space $(X, d, \mu)$ and a function
 $f\in L^1(\mu)$, such that 
$Mf (x) <  |f|(x)$ on a set of positive measure.  In fact, $X$ can be taken to be an infinite
dimensional  separable Hilbert space,
and $\mu$ a Gaussian measure on $X$.
\end{corollary}

\begin{proof} By \cite{Pr}, there exists a Gaussian (probability) measure $\gamma$ on a separable Hilbert space $H$
and a Borel subset $J$ with $\gamma J < 1$, such that $\gamma$-a.e. x,   
$$
\lim_{r\downarrow 0} \frac{\gamma (J \cap B (x, r))}{\gamma
(B (x, r))} = 1.
$$
Next select a compact $F\subset J^c$ with   $\gamma F > 0$ (recall that in complete
metric spaces, compact is equivalent to closed and totally bounded). 
Since 
$$ \gamma (J \cap B (x, r)) +  \gamma (F \cap B (x, r))  \le \gamma 
(B (x, r)),
$$
 for almost every $x$ we have 
$$
 \lim_{r\downarrow 0} \frac{\gamma (F \cap B (x, r)) }{\gamma
(B (x, r))} 
= 0. 
$$
\end{proof}

It follows from the proof of Theorem \ref{equivalence} that in fact
we can take $f = \mathbf{1}_F$, where $F$ is the compact set considered
in the preceding proof, so $M \mathbf{1}_F  <  \mathbf{1}_F $ on a set of positive measure.
Actually, by considering the three possibilities studied in the said proof (the supremum
is achieved when the radius tends to infinity, when it tends to zero, or for some fixed
radius $r$) it immediately follows from Preiss' result that 
for every $x \in F$, $M \mathbf{1}_F  (x) <  1$.

\

We present next  some positive results.
Recall that the symmetric difference between two sets is defined as $A\triangle B := (A\setminus B) \cup (B\setminus A)$. The next result is 
obtained via  a standard adaptation of the usual arguments to the case where  the $\limsup$ is considered, instead of the $\lim$.

\begin{lemma} {\bf From coverings to upper densities.}   \label{weak}  Let $(X, d, \mu)$ be a metric measure  space.
Suppose that for every closed and totally bounded set $F\subset \operatorname{supp} \mu$ and all $R, \varepsilon >0$,
there exists a finite disjoint collection of balls $\{B^{cl} (x_1, r_1), \dots , B^{cl} (x_N, r_N)\}$
such that

1) For $i= 1, \dots , N$, $x_i \in F$ and $0 < r_i < R$;  

2) $\mu (F\triangle \cup_1^N B^{cl} (x_i, r_i)) < \varepsilon$.

Then 
$$\limsup_{r\downarrow 0} \frac{1}{\mu
(B^{cl}(x, r))} \int _{B^{cl}(x, r )}  \mathbf{1}_F\ d\mu = \mathbf{1}_F (x)
$$
almost everywhere.
\end{lemma}

\begin{proof} Trivially, 
 for every $x \in F^c$, 
$$\lim_{r\downarrow 0} \frac{\mu (F \cap B^{cl} (x, r))}{\mu
(B^{cl} (x, r))} = 0,
$$
 since $F$ is closed. Thus, given $\delta \in (0,1)$, it is enough to show that  
$$
\limsup_{r\downarrow 0} \frac{1}{\mu (B^{cl}(x, r))} \int _{B^{cl}(x, r )}  \mathbf{1}_F\ d\mu
 \ge 
 \mathbf{1}_F (x) - \delta,
$$
 save perhaps on a set $A \subset F$ of measure $\le \delta$. Towards a contradiction, suppose 
$\mu A > \delta$ and for all $x\in A$,
$$
\limsup_{r\downarrow 0} \frac{1}{\mu (B^{cl}(x, r))} \int _{B^{cl}(x, r )}  \mathbf{1}_F\ d\mu
 < 1 - \delta.
$$
Now for each $x\in A$ there is an $r(x) > 0$ such that whenever $ r \le r(x)$, 
 $$
 \int _{B^{cl}(x, r )}    \mathbf{1}_F\ d\mu
 < (1 - \delta) \mu (B^{cl}(x, r)).
$$ 
Setting $A_n := \{x\in A: r(x) \ge 1/n\}$, we have that $\mu A = \lim_n \mu A_n$, so we can
choose an $L> 0$ such that $\mu A_L > \delta$. Furthermore, 
by Lemma \ref{inner}, we may suppose that $A_L$ is closed. Now fix $\varepsilon >0$
(with $\varepsilon \ll \delta$). By 
hypothesis, 
there exists a finite disjoint collection of balls $\{B^{cl} (x_1, r_1), \dots , B^{cl} (x_N, r_N)\}$
centered at $A_L$, with radii bounded by $1/L$, 
such that
$\mu (A_L\triangle \cup_1^N B^{cl} (x_i, r_i)) < \varepsilon$.
But now, since $A_L \subset F$,
$$
\mu A_L =   \int_{A_L}  \mathbf{1}_F\ d\mu
 \le  
\int_{ \cup_1^N B^{cl} (x_i, r_i) \cup A_L}     \mathbf{1}_F\ d\mu
< \varepsilon +  \int_{ \cup_1^N B^{cl} (x_i, r_i)}     \mathbf{1}_F\ d\mu
$$
$$
=
 \varepsilon + \sum_1^N  \int_{ B^{cl} (x_i, r_i)}     \mathbf{1}_F\ d\mu
< \varepsilon +  (1 - \delta)   \sum_1^N  \mu { B^{cl} (x_i, r_i)}
< 2 \varepsilon +    (1 - \delta) \mu A_L.
$$
Letting $\varepsilon \downarrow 0$, we obtain a contradiction.
\end{proof}

\begin{definition} \label{geomdoub} A metric space is {\it geometrically doubling}  if there exists a positive
integer $D$ such that every ball of radius $r$ can be covered with no more than $D$ balls
of radius $r/2$. 
\end{definition}

Geometrically doubling metric spaces are separable, so the following result applies to every Borel measure
(finite on bounded sets) on a geometrically doubling space. We mention that beyond the
case of measures that satisfy a local comparability condition (where the answer is positive, cf. \cite[Theorem 5.1]{Al})
little is known regarding the boundedness of the maximal operator for Borel measures on geometrically
doubling spaces.

\begin{corollary}\label{geomdoubling} Let $(X, d, \mu)$ be a  geometrically doubling metric measure space. Then
for all $f\in L_{loc}^1(\mu)$,
$Mf (x) \ge |f|(x)$ almost everywhere.
\end{corollary}

\begin{proof} This follows from Hyt\"onen's generalization (cf. \cite[Lemma 3.3]{Hy}) to the geometrically doubling setting,
of Tolsa's Lemma on the existence of arbitrarily small doubling cubes in $\mathbb{R}^d$  (cf. \cite[Lemma 2.8]{To}).

By  Hyt\"onen's result, there exists a constant $C > 1$ such that for $\mu$-a.e $x$ one can find  a sequence of radii $r_n(x)$ with
$\mu (B^{cl}(x, 3 r_n(x)) \le C \mu B^{cl}(x, r_n(x))$ and $r_n(x) \downarrow 0$ as $n\to \infty$. 
Let  $F \subset \operatorname{supp} \mu$ be a closed and totally bounded set of positive measure, 
and fix $\varepsilon > 0$. By removing a set of measure zero if needed, we may assume that
 all points in $F$ are centers of decreasing sequences of small doubling balls. 
Choose $N \gg 1$ so that  the $1/N$-neighborhood of $F$, denoted here by
$Bl(F, 1/N):= \cup_{x\in F }B(x,1/N)$, satisfies 
$\mu (Bl(F, 1/N) \setminus F) < \varepsilon/ 2$. This can always be done since
$F = \cap_{n\ge 1} Bl(F, 1/n)$.
Using the standard argument for doubling measures we obtain a finite disjoint 
 collection of balls $\{B^{cl} (x_1, r_1), \dots , B^{cl} (x_L, r_L)\}$
such that
for $i= 1, \dots , L$, $x_i \in F$, $0 < r_i < 1/N$, and 
 $\mu (F\setminus \cup_1^L B^{cl} (x_i, r_i)) < \varepsilon/2$.
In this part of the proof the fact that the balls are closed is used: Given any $x\in F$ not already covered,
it is always  possible to find a ball $B^{cl} (x, r)$, disjoint from the balls previously selected,  and with 
$r < 1/N$. Furthermore,  since
all chosen balls are contained in $Bl(F, 1/N)$, we conclude that
$\mu (\cup_1^L B^{cl} (x_i, r_i) \setminus F) < \varepsilon/2$.
Now the result follows from Lemma \ref{weak} and Theorem \ref{equivalence}.
\end{proof}

The next corollary assumes that the measure of balls does not decrease too fast with the radius:
There are ``local polynomial bounds" to such reduction.

\begin{corollary}\label{small} Let $(X, d, \mu)$ be a  metric measure space.
Suppose for every closed and totally bounded set $F\subset \operatorname{supp} \mu$ with $\mu F > 0$,
there are two functions $\psi_F, \phi_F > 0$ on $F$, and a  constant $c_F > 0$,
 such that  for all $x\in F$ we have  $\phi_F (x) \le 1$  and
$\mu B(x, r) \ge \psi_F (x) r^{c_F}$ whenever $r \le \phi_F(x)$. 
 Then
for all $f\in L_{loc}^1(\mu)$,
$Mf (x) \ge |f|(x)$ almost everywhere.
\end{corollary}

\begin{proof} Select a closed and totally bounded set $F\subset \operatorname{supp} \mu$ with $\mu F > 0$.
To apply Lemma \ref{weak}, it is enough to show that every point $x\in F$ is the center of arbitrarily small doubling balls.
So choose a small radius $r \le \phi_F(x)$ and towards a contradiction,  suppose $B^{cl } (x, r)$ does not contain any ball of the form
 $B^{cl } (x, 3^{-j-1} r)$ satisfying $\mu B^{cl } (x, 3^{-j} r) \le 4^{c_F} \mu  B^{cl } (x, 3^{-j-1} r)$.
Then $\mu B^{cl } (x, r) >  4^{c_F} \mu  B^{cl } (x, 3^{-1} r) > \cdots >  4^{ j c_F}  \mu B^{cl } (x, 3^{-j} r) 
\ge ( 4/3)^{ j c_F}   r^{ c_F}  \psi_{F} (x) \uparrow \infty$ as $j\to \infty$. 
\end{proof}

\begin{example} In the preceding corollary we make $\psi$, $\phi$ and $r$ depend on the closed and totally bounded
set $F$, rather than taking them uniform over the whole space. This way the corollary can be applied,
for instance, to spaces that are constructed putting together, say, manifolds of
different dimensions. The role of $\phi$ is to avoid assuming hypotheses  that imply
fast growth of balls with large radii.

Next
we give an example where a
uniform   $c$
does not exist.
Let $e_{1, n}$ denote the first vector of the standard  basis in $\mathbb{R}^n$, 
 let $B^n(3n e_{1, n}, 1) := \{x\in \mathbb{R}^n: \|x - 3n e_{1, n}\|_2 < 1\}$, and let 
 $X$ be the disjoint union $X := \cup_{n=6}^\infty B^n(3n e_{1, n}, 1)$.
To define a metric on $X$, within each ball $B^n(3n e_{1, n}, 1)$ we use the standard euclidean distance,
while if $x\in B^n(3n e_{1, n}, 1)$ and $y\in B^m(3m e_{1, m}, 1)$, with $n < m$, then
$d(x,y) = 3m - 3n + \|x - 3n e_{1, n}\|_2 + \|y - 3m  e_{1, m}\|_2$.
Finally, $\mu$ restricted to   $B^n(3n e_{1, n}, 1)$ is simply $n$-dimensional Lebesgue measure. 
 Then there is no uniform $c$ that works over the
whole space, but given $F\subset X$ compact, it can intersect at most finitely many components of $X$,
so there is a ball of highest dimension, say $ B^N (3N e_{1, N}, 1)$,with $F\cap  B^N(3N e_{1, N}, 1)\ne \emptyset$.
Then we can take $\phi_F =1$,  and $\psi_F (x) r^{c_F} = V_N r^N$, where $V_N$ denotes the volume of the $N$-dimensional 
 unit ball (which decreases for $n \ge 6$).
\end{example}

Finally, we mention that positive results for
some classes of Gaussian measures on Hilbert spaces  were already known.  The following theorem is given in
\cite{Ti}:

\begin{theorem} Let $H$ be a Hilbert space and let $\gamma$ be a Gaussian measure with the following spectral representation of its covariance operator:
$$
Rx = \sum c_i (x, e_i) e_i
$$
where $(e_i)$ is an orthonormal system in $H$. Suppose $c_{i +1} \le c_i i^{-\alpha}$  for 
a given $\alpha  > 5/2$. Then for all $f \in L^p (\gamma)$, where  $p> 1$,  and almost every $x$,
we have 
$$
\lim_{r\downarrow 0} \frac{1}{\gamma
( B (x, r))} \int _{ B (x, r )}  |f  - f(x)| \ d\gamma  = 0.
$$
\end{theorem}

Since $\mathbf{1}_F$ belongs to all $L^p$ spaces, we conclude that for all $f\in L^1(\gamma)$,
$Mf (x) \ge |f|(x)$ almost everywhere. Actually, there is a previous result due to
Preiss and Ti\u{s}er (cf. \cite{PrTi})  which under a weaker hypothesis 
(the lacunarity of the  $c_i$, that is, there exists a $q < 1$ such that for all $i \ge 1$,
 $c_{i +1} \le c_i q$) 
yields the convergence in
measure of the averages $
 \frac{1}{\gamma
( B (x, r))} \int _{ B (x, r )}  f \ d\gamma$, as $r\downarrow 0$, to $f\in L^1(\gamma)$. Now convergence in measure entails that
some subsequence converges a.e., and this is enough to ensure that 
$Mf (x) \ge |f|(x)$ almost everywhere.

\section{The case of larger $\sigma$-algebras}

It is a natural follow-up question to study what
happens with $\sigma$-algebras that are larger than the Borel sets. 
 Let $(X, d, \mu)$ be a  metric measure space, and let
$\nu$ be an extension of $\mu$ to a $\sigma$-algebra $\mathcal{A}$  properly containing the Borel
sets of $X$. Now if $L^1_{loc}(\mu) = L^1_{loc}(\nu) $, the characterization
given in Theorem \ref{equivalence} still holds, as can be seen from the following
argument. On the one hand, the class of totally bounded
closed sets does not change, and on the other, given any locally integrable $\nu$-measurable
function $f$, by hypothesis there is a Borel function $g$ such that
$\nu$-a.e., $g = f$ ;  since $\nu$ is an extension
of $\mu$, all the $\nu$-averages of $g$ equal its $\mu$-averages, so
$M_\mu g (x) \ge |g|(x)$ $\mu$-a.e. $x$ iff   
$M_\nu g (x) \ge |g|(x)$ $\nu$-a.e. $x$ iff   
$M_\nu f (x) \ge |f|(x)$ $\nu$-a.e. $x$.

In this paper we deal with locally finite measures on metric spaces, so they are
automatically $\sigma$-finite. Following \cite[Page 88, footnote 2]{AlMi}, we say that
$\mu$ is {\em Borel semiregular} if for every $A \in \mathcal{A}$, there
is a Borel set $B$ such that $\nu (A\triangle B) = 0$. Also,
we say that
$\mu$ is {\em Borel regular} if for every $A \in \mathcal{A}$, there
exist  Borel sets $B, C$ such that $B \subset A \subset C$ and $\mu (C\setminus B) = 0$
(in the present context this is equivalent to \cite[Definition 2.9.2]{AlMi}). Thus, Borel
semiregularity is more general than Borel regularity. 
Under Borel semiregularity,
$\mathcal{A}$-measurable simple functions have Borel representatives, so it follows
that  $L^1_{loc}(\mu) = L^1_{loc}(\nu) $ and the preceding considerations apply.

It is shown in \cite[Theorem 3.14]{AlMi} that if one starts with a doubling
measure and extends it beyond the measurable sets, the Lebesgue differentiation
theorem holds for the larger class of functions precisely when the extension
is Borel semiregular. We mention that on geometrically doubling metric spaces,
this result also holds for the class (more general than the doubling measures)
of Borel measures
satisfying a local comparability condition.

Without Borel semiregularity the connection between averages of measurable
functions, and averages  of indicator functions of closed sets, breaks down. This is probably best
explained through an example. Also,  it seems interesting to exhibit instances of measures that are Borel semiregular but not regular, and of
measures that are not Borel semiregular.

\begin{example} Given a  metric measure space  $(X, d, \mu)$, the natural, Borel
regular extension of $\mu$, which we also denote by $\mu$, is defined on  the $\sigma$-algebra 
of $\mu^*$-measurable sets, via the Caratheodory's construction. It can happen that
in this way
all sets become measurable, as is the case, for instance, when the measure is a Dirac delta,
so further extensions may be impossible. 

Even when such extensions are possible, moving
into the realm of non-measurable sets
will typically  require the full force of the Axiom of Choice: By a well known result of R. Solovay,
cf. \cite{So}, under the assumption that there exists an inaccessible cardinal, the ZF axioms, 
together with the Axiom of Dependent
Choice,  and the assumption that every set of reals is Lebesgue measurable, form a consistent
theory.

Taking into account these considerations, let   $\lambda$ be the Lebesgue measure on $X=[0,1]$, with the usual distance.
Let $E\subset [0,1]$ be such that both $\lambda^* E = 1 =  \lambda^* E^c$
(the existence of such sets is a well known consequence of the Axiom of Choice). 
Define $\lambda_1$ on the Borel sets of $E$ by setting, for each Borel set 
$B \subset [0,1]$,  $\lambda_1 E\cap B = \lambda B$. Note that $\lambda_1$
is well defined: if $B_1$ and $B_2$ are Borel sets such that $E \cap B_1 = E \cap B_2$, then  $B_1$ and $B_2$ have the same Lebesgue measure, as can be
seen from the following argument. Since 
$E \cap B_1 = E \cap B_2 = E \cap (B_1\cap B_2)$, 
$B_1 \setminus (B_1\cap B_2) \subset E^c$, so 
$\lambda(B_1 \setminus (B_1\cap B_2)) = 0$. Likewise 
$\lambda(B_2 \setminus (B_1\cap B_2)) = 0$, so 
$\lambda(B_1) = 
\lambda (B_1\cap B_2)
= 
\lambda (B_2).$

Next, define 
$\lambda_2$ on the Borel sets of $E^c$ by setting, for each Borel set 
$B \subset [0,1]$,  $\lambda_2 E^c \cap B = \lambda B$.
For every $t\in [0,1]$ set $\mu_t := t\lambda_1 + (1 -t) \lambda_2$. 
Then it is easy to check that both $\mu_0$ and $\mu_1$ are Borel 
semiregular but not Borel regular, while for $0 < t < 1$, $\mu_t $ is not Borel semiregular, and
the condition 
$M_{\mu _t} \mathbf{1}_E  (x) \ge  \mathbf{1}_E  (x)$ $\mu_t$-a.e. $x$ fails.
In particular, for $t = 1/2$ we have 
$M_{\mu_{1/2}} \mathbf{1}_E   \equiv 1/2  \equiv M_{\mu_{1/2}} \mathbf{1}_E^c$. 
\end{example}

The situation observed in the preceding example on Lebesgue measure is typical  for
doubling measures, and more generally, for measures that satisfy a local comparability
condition and are defined on a geometrically doubling metric space, since in this case
the Vitali Covering Theorem is available (cf. \cite[Theorem 2.8]{Al}). 

More precisely, let $\nu$ be a Borel measure, and let 
$\mu$ be an  extension of $\nu$ to the $\sigma$-algebra $\mathcal{A}$,
such that $\mu$ fails to be Borel semiregular. Let $E\in \mathcal{A}$ witness this failure, so
for every Borel set $B$, $\mu (E\triangle B) > 0$. By arguing as in the proof of 
Theorem \ref{equivalence}, it is enough to show that 
$$\limsup_{r\downarrow 0} \frac{1}{\mu
(B^{cl}(x, r))} \int _{B^{cl}(x, r )}  \mathbf{1}_E\ d\mu <  \mathbf{1}_E (x)
$$
on a set of positive $\mu$-measure. Choose a decreasing sequence of open sets $O_n$ such
that 
$\nu^*(E) = \lim_n \nu O_n$.
Now by the failure of Borel semiregularity it follows that
$\inf_n \mu (E^c \cap O_n) > 0$. Taking a Vitali covering of $E$ by centered closed balls inside
$O_n$, we can select a disjoint collection $\{B^{cl}_i\}_{1 \le i < L\le \infty}$
with $\nu^*(E \setminus \cup_{1 \le i < L} B^{cl}_i) = 0$. If we had 
$$\limsup_{r\downarrow 0} \frac{1}{\mu
(B^{cl}(x, r))} \int _{B^{cl}(x, r )}  \mathbf{1}_E\ d\mu = \mathbf{1}_E (x)
$$
a.e., we would be able to select  the disjoint collection 
$\{B^{cl}_i\}_{1 \le i < L\le \infty}$
with $\mu (E^c \cap (\cup_{1 \le i < L} B^{cl}_i))$ as small as we wanted,
contradicting  $\inf_n \mu (E^c \cap O_n) > 0$.

\end{document}